\documentclass[preprint,12pt]{elsarticle}
\usepackage[T1]{fontenc}                        
\usepackage{mathtools,amsthm,amssymb}           
\usepackage{fullpage}
\usepackage{hyperref}
\usepackage{theoremref}
\usepackage{physics}
\usepackage{enumitem}
\usepackage{float}
\usepackage{subcaption}

\usepackage{tikz}
\usetikzlibrary{calc, angles}

\usepackage{pgfplots}
\pgfplotsset{compat=1.13}

\journal{Linear Algebra Appl.}

\newcommand{\karp}{Karpelevi{\v{c}}}
\newcommand{\ii}{\mathsf{i}}	
\DeclareMathOperator{\Arg}{Arg}
\DeclareMathOperator{\conv}{conv}
\newcommand{\floor}[1]{\left\lfloor #1 \right\rfloor}

\newtheorem{theorem}{Theorem}[section]
\newtheorem{lemma}[theorem]{Lemma}

\newtheorem{proposition}[theorem]{Proposition}
\newtheorem{conjecture}[theorem]{Conjecture}

\theoremstyle{definition}
\newtheorem{definition}[theorem]{Definition}
\newtheorem{observation}[theorem]{Observation}
\newtheorem{remark}[theorem]{Remark}
\newtheorem{example}[theorem]{Example}

\begin{document}
\begin{frontmatter}
    \title{Demystifying the \karp~theorem}
    \author{Devon N.~Munger}
    \author{Andrew L.~Nickerson}
        \ead{nickera@wwu.edu}
    \author[add]{Pietro Paparella\corref{cor1}}
        \ead{pietrop@uw.edu}
        \affiliation[add]{organization={Division of Engineering \& Mathematics, University of Washington Bothell},
                    addressline={18115 Campus Way NE}, 
                    city={Bothell},
                    postcode={98011-8246}, 
                    state={WA},
                    country={U.S.A}}
        \cortext[cor1]{Corresponding author.}

\begin{abstract}
    The statement of the \karp~theorem concerning the location of the eigenvalues of stochastic matrices in the complex plane (known as the \karp~region) is long and complicated and his proof methods are, at best, nebulous. Fortunately, an elegant simplification of the statement was provided by Ito---in particular, Ito's theorem asserts that the boundary of the \karp~region consists of arcs whose points satisfy a polynomial equation that depends on the endpoints of the arc. Unfortunately, Ito did not prove his version and only showed that it is equivalent.
    
    More recently, Johnson and Paparella showed that points satisfying Ito's equation belong to the \karp~region. Although not the intent of their work, this initiated the process of proving Ito's theorem and hence providing another proof of the \karp~theorem.  
    
   The purpose of this work is to continue this effort by showing that an arc appears in the prescribed sector. To this end, it is shown that there is a continuous function $\lambda:[0,1] \longrightarrow \mathbb{C}$ such that $\mathsf{P}^\mathsf{I}(\lambda(\alpha)) = 0$, $\forall \alpha \in [0,1]$, where $\mathsf{P}^\mathsf{I}$ is a Type I reduced Ito polynomial. It is also shown that these arcs are simple. Finally, an elementary argument is given to show that points on the boundary of the \karp~region are extremal whenever $n > 3$.  
\end{abstract}

\begin{keyword}
    \karp~arc \sep \karp~region \sep stochastic matrix

    \MSC[2020] 15A18 \sep 15B51 \sep 30C15
\end{keyword}

\end{frontmatter}

\section{Introduction}

A \emph{stochastic} matrix is an entrywise nonnegative matrix whose rows sum to one. In 1938\footnote{Many authors misattribute Kolmogorov's proposal to a 1937 paper, which is perhaps due to a misattribution by Gantmacher in 1959 (for full details, see Swift \cite[p.~2]{s1972}).}, Kolmogorov posed the problem of characterizing the subset of the complex plane, which we denote by $\Theta_n$, that comprises all eigenvalues of all $n$-by-$n$ \emph{stochastic} matrices \cite[p.~2]{s1972}. 

Dmitriev and Dynkin \cite{dd1946} (see \cite{ams140} for an English translation) obtained a partial solution, and \karp~\cite{k1951} (see \cite{ams140} for an English translation) solved the problem by showing that the boundary of $\Theta_n$ consists of curvilinear \emph{arcs} (hereinafter, \emph{\karp~arcs} or \emph{K-arcs}), whose points satisfy a polynomial equation that depends on the endpoints of the arc. 

For $n \in \mathbb{N}$, let $F_n \coloneqq  \{ p/q \mid 0\leq p \le q \leq n,~\gcd(p,q)=1 \}$. The following is an equivalent form of \karp's theorem due to Ito. 

\begin{theorem} 
        \thlabel{karpito}
    The region $\Theta_n$ is symmetric with respect to the real axis, is included in the unit-disc $\{ z \in \mathbb{C} \mid |z| \leq 1\}$, and intersects the unit-circle $\{ z \in \mathbb{C} \mid |z| = 1\}$ at the points $\left\{ e^{\frac{2\pi p}{q}\ii} \mid p/q \in F_n \right\}$. The boundary of $\Theta_n$ consists of these points and of curvilinear arcs connecting them in circular order. 
    
    Let the endpoints of an arc be $e^{\frac{2\pi p}{q}\ii}$ and $e^{\frac{2\pi r}{s}\ii}$ ($q \le s$). Each of these arcs is given by the following parametric equation:  
        \begin{equation}
            \label{itoequation}
                t^{s} \left( t^{q} - \beta \right)^{\floor{n/q}} = \alpha^{\floor{n/q}} t^{q\floor{n/q}},\ \alpha \in [0,1], ~\beta\coloneqq 1-\alpha.
        \end{equation} 
\end{theorem}

Unfortunately, Ito did not prove \thref{karpito} and only showed that the statement above is equivalent to the statement given by \karp~\cite[p.~80, Theorem B]{ams140}. The statement by \karp~is long and cumbersome and the arguments and methods he employed to prove his theorem are, at best, nebulous.  

As such, an elementary proof of \thref{karpito} would place \karp's work on a firmer foundation and is of interest to researchers---especially those working on the celebrated \emph{nonnegative inverse eigenvalue problem} \cite{l2019}. 

Following Johnson and Paparella \cite{jp2017}, equation \eqref{itoequation} is called the \emph{Ito equation} and the polynomial
    \begin{equation}
        \label{itopolynomial}
            \mathsf{P}_\alpha(t)\coloneqq  t^s(t^q-\beta)^{\floor{n/q}}-\alpha^{\floor{n/q}}t^{q\floor{n/q}},\ \alpha\in[0,1],\ \beta \coloneqq 1-\alpha    
    \end{equation}
is called the \textit{Ito polynomial}.

In addition to establishing the basic properties listed in the first paragraph of \thref{karpito} (which are proven in Section \ref{sec:notandback}), proving the following result would yield a proof of \thref{karpito}.

\begin{proposition}
    \thlabel{bigprop}
        Suppose that $n > 3$ and $p/q$ and $r/s$ are Farey neighbors.
                \begin{enumerate}[label=\roman*)]
                \item \label{firststep} If $\zeta$ satisfies equation \eqref{itoequation}, then $\zeta \in \Theta_n$.
                \item \label{secondstep} There is a continuous function $\lambda: [0,1] \longrightarrow \mathbb{C}$ such that $\lambda(0) = e^{\frac{2\pi p}{q} \ii}$, $\lambda(1) = e^{\frac{2\pi r}{s} \ii}$, and $\mathsf{P}_\alpha (\lambda(\alpha)) = 0, \forall \alpha \in [0,1]$. Furthermore, 
                \[ \min \left\{ \frac{p}{q}, \frac{r}{s} \right\} \le \frac{\Arg \lambda(\alpha)}{2\pi} \le \max \left\{ \frac{p}{q}, \frac{r}{s} \right\}, \forall \alpha \in [0,1]. \]
                \item \label{thirdstep} If $\lambda$ is a function satisfying the properties listed in part \ref{secondstep}, then $\lambda(\alpha) \in \partial \Theta_n$, $\forall \alpha \in [0,1]$.
            \end{enumerate}
\end{proposition}

Part \ref{firststep} of \thref{bigprop} was proven by Johnson and Paparella \cite{jp2017} as follows: suppose that $\mathsf{P}_\alpha(\zeta) = 0$, with $\zeta \ne 0$. Note that $s \ne q\floor{n/q}$ since $\gcd{(q,s)} = 1$. 
\begin{itemize}
    \item If $\floor{n/q} = n$, then $\zeta$ is a zero of the \emph{Type 0 (Ito) polynomial}
        \begin{equation}
            \label{type_zero_ito}
                \mathsf{P}_\alpha^{\mathsf{0}} (t) = (t-\beta)^n - \alpha^n,~\alpha\in[0,1],\ \beta \coloneqq 1-\alpha.
        \end{equation}
    The Type 0 polynomial corresponds the \emph{Farey pair} $(0/1,1/n)$.

    \item If $\floor{n/q} = 1$, then $\zeta$ is a zero of the \emph{Type I (Ito) polynomial} 
        \begin{equation}
            \label{type_one_ito}
                \mathsf{P}_\alpha^{\mathsf{I}}(t) = t^s-\beta t^{s-q} - \alpha,\ \alpha\in[0,1],\ \beta \coloneqq 1-\alpha.
        \end{equation}
    
    \item If $1 < \floor{n/q} < n$ and $s > q\floor{n/q}$, then $\zeta$ is a zero of the \emph{Type II (Ito) polynomial}
        \begin{equation}
            \label{type_two_ito}
                \mathsf{P}_\alpha^{\mathsf{II}} (t) = (t^q-\beta)^{\floor{n/q}}-\alpha^{\floor{n/q}}t^{q\floor{n/q}-s},\ \alpha\in[0,1],\ \beta \coloneqq 1-\alpha.
        \end{equation}
    
    \item If $1 < \floor{n/q} < n$ and $s < q\floor{n/q}$, then $\zeta$ is a zero of the \emph{Type III (Ito) polynomial}
        \begin{equation}
            \label{type_tre_ito}
                \mathsf{P}_\alpha^{\mathsf{III}} (t) = t^{s-q\floor{n/q}}(t^q-\beta)^{\floor{n/q}}-\alpha^{\floor{n/q}},\ \alpha\in[0,1],\ \beta \coloneqq 1-\alpha. 
        \end{equation}
\end{itemize}
The polynomials given by equations \eqref{type_zero_ito}--\eqref{type_tre_ito} are called the \emph{reduced Ito polynomials}. Johnson and Paparella \cite[Theorem 3.2]{jp2017} showed that for every $\alpha \in [0,1]$, there is a stochastic matrix $A = A(\alpha)$ such that $\chi_A(t) = \mathsf{P}_\alpha^\mathsf{X}(t)$, where $\mathsf{X} \in \{ \mathsf{0},\mathsf{I},\mathsf{II}, \mathsf{III} \}$ and $\chi_A$ denotes the characteristic polynomial of $A$ (for further results, see Kirkland and \v{S}migoc \cite{ks2022}). 

The following remarkable result is due to Kato \cite[Theorem 5.2]{k1995}.

\begin{theorem}[Kato] 
    \thlabel{kato}
        Let $\Lambda(\alpha)$ be an unordered $N$-tuple of complex numbers, depending continuously on a real variable $\alpha$ in a (closed or open) interval $I$. Then there exist $N$ single-valued, continuous functions $\lambda_k(\alpha), k = 1,\ldots,N$, the values of which constitute the $N$-tuple $\Lambda(\alpha)$ for each $\alpha \in I$.  
\end{theorem}

If $\lambda: [0,1] \longrightarrow \mathbb{C}$ and $\lambda$ is continuous, then $\lambda$ is called an \emph{arc} (or \emph{path}). Because the zeros of a polynomial depend continuously on its coefficients (see, e.g., Horn and Johnson \cite[Appendix D]{hj2013}), by \thref{kato}, there are, including repetitions, $N \coloneqq \deg \mathsf{P}_\alpha$ arcs $\lambda_k: [0,1] \longrightarrow \mathbb{C}$, $k \in \{1,\ldots,N \}$, such that $\mathsf{P}_\alpha(\lambda_k(\alpha)) = 0$, $\forall \alpha \in [0,1]$. 

With the above in mind, it is possible to give a precise definition of what is meant by a \emph{K-arc}. For $n \in \mathbb N$, let $\omega_n \coloneqq  \cos(2\pi/n) + \ii \sin(2\pi/n)$.

\begin{definition}
    Let $n \ge 4$, $\mathsf{X} \in \{ \mathsf{0},\mathsf{I},\mathsf{II}, \mathsf{III} \}$, and let $\frac{p}{q}$ and $\frac{r}{s}$ be Farey neighbors. If there is a continuous function $\lambda: [0,1] \longrightarrow \mathbb{C}$ such that $\lambda(0) = \omega_q^p$, $\lambda(1) = \omega_{s}^{r}$, $\mathsf{P}_\alpha^\mathsf{X} (\lambda(\alpha)) = 0, \forall \alpha \in [0,1]$, and  
    \[ \min \left\{ \frac{p}{q}, \frac{r}{s} \right\} \le \frac{\Arg \lambda(\alpha)}{2\pi} \le \max \left\{ \frac{p}{q}, \frac{r}{s} \right\}, \forall \alpha \in [0,1], \] 
    then 
    \[ K_n\left( \min \left\{ \frac{p}{q}, \frac{r}{s} \right\}, \max \left\{ \frac{p}{q}, \frac{r}{s} \right\} \right) \coloneqq \left\{ \lambda \in \mathbb{C} \mid \lambda = \lambda(\alpha),\ \alpha \in [0,1] \right\} \]
    is called the \emph{K-arc with respect to $\frac{p}{q}$ and $\frac{r}{s}$ (of order $n$)}.
\end{definition} 

Part \ref{secondstep} of \thref{bigprop} is trivial for the Type 0 polynomial (see Subsection \ref{subsect:typezero}). Although not the central aim of his work, {\DJ}okovi{\'c} \cite[pp.~175--181]{d1990} indirectly established part \ref{secondstep} of \thref{bigprop} for Type I polynomials by examining boundaries of so-called \emph{Farey tilings}. In an effort to sharpen the \karp~theorem, Kirkland et al.~\cite[Theorem 4.2]{kls2020} also indirectly established this result. In particular, given $\theta \in [0,2\pi)$, Kirkland et al~\cite[Theorem 1.2]{kls2020} gave an explicit description of the point on the boundary of $\Theta_n$ with argument $\theta$. However, the \karp~theorem is taken as their starting point; moreover, the modulus of this point is given as the zero of a separate polynomial. In both of the aforementioned works, the arguments are somewhat lengthy and involved. 

In what follows, a direct and elementary argument is given to establish part \ref{secondstep} of \thref{bigprop} for Type I polynomials (see Subsection \ref{subsect:typeone}). Although this is a special case, many Type II and Type III arcs of a given order are pointwise powers of certain Type I arcs (see Section \ref{sect:karcpowers}). Part \ref{secondstep} of \thref{bigprop} for Type II and Type III arcs that are not pointwise powers is left open.

If $\lambda \in \Theta_n$, then $\lambda$ is called \emph{extremal} if $\gamma \lambda \notin \Theta_n$, $\forall \gamma  > 1$. If $E_n$ denotes the collection of extremal numbers, then it is clear that $E_n \subseteq \partial\Theta_n$. \karp~\cite[p.~81]{ams140} asserted that $\partial \Theta_n = E_n$ because $\Theta_n$ is closed. The case when $n=1$ is trivial, but \karp's assertion is false when $n=2$ and $n=3$---indeed, notice that $\partial \Theta_2 \setminus E_2 = (-1,1)$ and $\partial\Theta_3 \setminus E_3 = (-1,-1/2]$ (see Figure \ref{fig:thetan}). In this work, it is also shown that $\partial\Theta_n \subseteq E_n$ when $n>3$ via elementary means.

\begin{figure}[H]
    \centering
        \begin{tikzpicture}
                \begin{axis}[
                    axis lines = none,
                    ticks = none,
                    axis equal image,
                    scale=0.75,
                    xmin=-1,
                    xmax=1,
                    ymin=-1.0,
                    ymax=1.0,
                    legend pos=outer north east
                    ]
        
                    \draw[color=gray] (axis cs:0,0) circle (1);
                    
                    \addplot[thick,dotted] coordinates{(-1,0) (1,0)};                              
                    \addlegendentry{$\Theta_2$};
                    
                    \addplot[thick,dashed] coordinates{
                    (1,0) 
                    (-.5,.866025403784439) 
                    (-.5,-.866025403784439)
                    (1,0)}; 
                    \addlegendentry{$\partial \Theta_3 \backslash \Theta_2$};
                \end{axis}
            \end{tikzpicture}
    \caption{The region $\Theta_2$ and the boundary of $\Theta_3$.}
    \label{fig:thetan}
\end{figure}
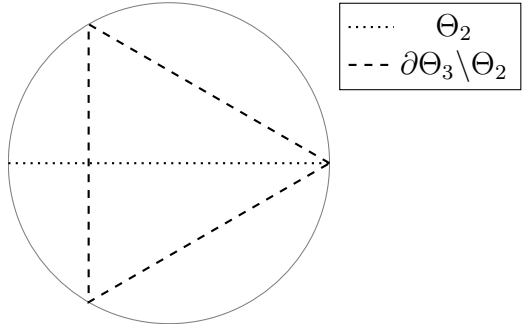
Thus, part \ref{thirdstep} of \thref{bigprop}, which is also left open, can be restated as follows: \emph{If $\lambda$ is a function satisfying the properties listed in part \ref{secondstep} of \thref{bigprop}, then $\lambda(\alpha) \in E_n$, $\forall \alpha \in [0,1]$.} 

\section{Notation \& Background}
\label{sec:notandback}

If $z$ is a nonzero complex number, then $\arg z \coloneqq  \{ \theta \in \mathbb{R} \mid z = r(\cos \theta + \ii\sin \theta) \}$ and $\Arg z \coloneqq  \{ \theta \in \arg z \mid \theta \in -(\pi,\pi] \}$. If $z_1,\ldots,z_m \in \mathbb{C}^n$, then the \emph{convex hull of $z_1,\ldots,z_m$}, denoted by $\conv(z_1,\ldots,z_m)$, is defined by 
\[ \conv(z_1,\ldots,z_m) = \left\{ \sum_{k=1}^n \alpha_k z_k \in \mathbb{C}^n \mid \sum_{k=1}^n \alpha_k = 1,\ \alpha_k \ge 0,\ 1 \le k \le m \right\}. \]

Given $n \in \mathbb{N}$, elements of the set $F_n \coloneqq  \{ p/q \mid 0\leq p \le q \leq n,\gcd(p,q) = 1 \}$ are called the \emph{Farey fractions of order n}. If $p/q$ and $r/s$ are elements of $F_n$ such that $p/q < r/s$, then $(p/q,r/s)$ is called a \emph{Farey pair (of order $n$)} if $x \not\in F_n$ whenever $p/q < x < r/s$. The Farey fractions $p/q$ and $r/s$ are called \emph{Farey neighbors} if $(p/q,r/s)$ or $(r/s, p/q)$ is a Farey pair. 

The following result is well-known (see, e.g., LeVeque \cite[Theorem 8.14]{l2002}). 

\begin{theorem} 
    \thlabel{leveque}
        If $p/q,r/s \in F_n$, then $p/q$ and $r/s$ are Farey neighbors of order $n$ if and only if $ps - qr = \pm 1$ and $q + s > n$.
\end{theorem} 

A \emph{directed graph} (or \emph{digraph}) $\Gamma = (V,E)$ consists of a finite, nonempty set $V$ of \emph{vertices}, together with a set $E \subseteq V^2$ of \emph{arcs}. For $A \in \mathsf{M}_{n}(\mathbb{C})$, the \emph{directed graph} (or \emph{digraph}) of $A$, denoted by $\Gamma(A)$, has vertices $V = \{ 1, \dots, n \}$ and arcs $E = \{ (i,j) \in V^2 \mid a_{ij} \neq 0\}$. 

A digraph $\Gamma$ is called \emph{strongly connected} if, for any two distinct vertices $i$ and $j$ of $\Gamma$, there is a directed walk in $\Gamma$ from $i$ to $j$ (following Brualdi and Ryser \cite{br1991}, every vertex of $V$ is considered strongly connected to itself). Let $k = k(\Gamma)$ be the greatest common divisor of the lengths of the closed directed walks of $\Gamma$ ($k$ is undefined when $n=1$ and $\Gamma$ does not contain a loop). A strongly connected digraph $\Gamma$ is called \emph{primitive} if $k = 1$ and \emph{imprimitive} if $k > 1$. 

If $n \geq 2$, then $A \in \mathsf{M}_{n}(\mathbb{C})$  is called \emph{reducible} if there is a permutation matrix $P$ such that
\begin{align*}
P^\top A P =
\begin{bmatrix}
A_{11} & A_{12} \\
0 & A_{22}
\end{bmatrix},
\end{align*}
where $A_{11}$ and $A_{22}$ are nonempty square matrices. If $A$ is not reducible, then A is called \emph{irreducible}. It is well-known that a matrix $A$ is irreducible if and only if $\Gamma(A)$ is strongly connected (see, e.g., Brualdi and Ryser \cite[Theorem 3.2.1]{br1991} or Horn and Johnson \cite[Theorem 6.2.24]{hj2013}). 

If $A$ is a nonnegative matrix of order $n$, then $A$ is primitive if and only if $\exists m \in \mathbb{N}$ such that $A^m > 0$. Furthermore, since $A$ is nonnegative and irreducible, it follows that $A^p > 0$, $\forall p \ge m$ (see, e.g., Brualdi and Ryser \cite[Theorem 3.4.4]{br1991} or Horn and Johnson \cite[Theorem 8.5.2]{hj2013}). If $A \in \mathsf{M}_n(\mathbb{R})$, then $A$ is called \emph{eventually positive} $\exists m \in \mathbb{N}$ such that $A^p > 0$, $\forall p \ge m$. It is known that $A$ is eventually positive if and only if the spectral radius of $A$ is a positive eigenvalue (of multiplicity 1) that strictly dominates all other eigenvalues and for which there is both a positive left and right eigenvector (see. e.g., Handelman \cite[Lemma 2.1]{h1981}, Johnson and Tarazaga \cite[Theorem 1]{jt2004}, or Noutsos \cite[Theorem 2.1]{n2006}). 

\section{Basic Results of \thref{karpito}}

\begin{proposition}
    \thlabel{realsymm}
        ``The region $\Theta_n$ is symmetric with respect to the real axis.''
\end{proposition}

\begin{proof}
    This is the least controversial aspect of \thref{karpito} given that the characteristic polynomial of a real matrix has real coefficients and thus nonreal zeros occur in complex conjugate pairs.
\end{proof}

\begin{proposition}
    \thlabel{modlessthanone}
        ``The region $\Theta_n$ is included in the unit-disc $\{ z \in \mathbb{C} \mid |z| \leq 1\}$.''
\end{proposition}

\begin{remark}
    A very elegant proof can be obtained by the Perron--Frobenius theorem as follows:  let $y$ be the left Perron vector normalized so that $y^\top e = 1$. By definition of the left Perron vector, $y^\top A = \rho(A) y^\top$ and because $Ae = e$, we have 
        $$1 = y^\top e = y^\top (Ae) = (y^\top A) e = (\rho(A) y^\top) e = \rho(A) (y^\top e) = \rho(A)\cdot 1 = \rho(A).$$  
        
    However, the following proof does not rely on nonnegativity. 
\end{remark}

\begin{proof}
    If $A \in \textsf{M}_n (\mathbb{C})$, then 
    $$\begin{Vmatrix} A \end{Vmatrix}_\infty \coloneqq  \max\limits_{1\le i \le n} \sum_{j=1}^n \vert a_{ij} \vert$$ 
    is a matrix norm \cite[Example 5.6.5]{hj2013}. If $A$ is stochastic, then $\begin{Vmatrix} A \end{Vmatrix}_\infty = 1$. Furthermore, if $\vert \vert \cdot \vert\vert$ a matrix norm, then $\rho(A) \le \begin{Vmatrix} A \end{Vmatrix}$ \cite[Theorem 5.6.9]{hj2013}. Thus, if $\lambda$ is an eigenvalue of a stochastic matrix, then $\vert \lambda \vert \le \rho(A) \le 1$.    
\end{proof}

\begin{proposition}
    ``The region $\Theta_n$ intersects the unit-circle $\{ z \in \mathbb{C} \mid |z| = 1\}$ at the points $\{ \omega_q^p \mid p/q \in F_n \}$.''
\end{proposition}

\begin{proof}
    This result was initially established by Romanovsky \cite{r1936}.
\end{proof}

\section{K-arcs}

In this section, part \ref{secondstep} of \thref{bigprop} is established for Type 0 and Type I polynomials.
    
\subsection{Type 0}\label{subsect:typezero}

It is known that if $\lambda = a + \mathsf{i}b$ is an eigenvalue of a stochastic matrix, then $a + \vert{b}\vert \tan(\pi/n) \le 1$ (for references, see Laffey \cite[p.~85]{l1995}). If $\mathsf{P}_\alpha^\mathsf{0} (t) \coloneqq  (t-\alpha)^n-\beta^n$, with $\alpha \in [0,1]$ and $\beta \coloneqq  1 - \alpha$, then $\mathsf{P}_\alpha^\mathsf{0} (\beta\omega_n + \alpha) = 0$. Thus, $\partial\Theta_n \cap \{ z \in \mathbb{C} \mid 0 \le \Arg z \le 2\pi/n \} = \conv{(1,\omega_n)}$. 

\subsection{Type I}\label{subsect:typeone}

\begin{observation}
    If $n \ge 4$ and $\floor{n/q} = 1$, then $q > 2$.
\end{observation}

\begin{proof}
For contradiction, if $q \le 2$, then  
\[ \floor{n/q}  = 1 \implies 1 \le \frac{n}{q} < 2, \]
i.e., $q \le n < 2q \le 4$. Thus, $n < 4$, which yields the desired contradiction.
\end{proof}

In what follows, it is assumed that $n \ge 4$ and $p/q$ and $r/s$ are Farey neighbors such that $\floor{n/q} = 1$ and 
\[ 0 < \min \left\{ \frac{p}{q},\frac{r}{s} \right\} < \max \left\{ \frac{p}{q},\frac{r}{s} \right\}  < \frac{1}{2}. \]

We start with some preliminary observations: Since $ps - qr = \pm 1$, it follows that $\gcd(q,s) = 1$. Thus, for the Type I polynomial,
\begin{equation}
    \label{typeonepoly}
        \mathsf{P}_\alpha^\mathsf{I} (t) = t^{s} - \beta t^{s-q} - \alpha,\ \alpha \in [0,1],\ \beta \coloneqq  1 - \alpha,\ \gcd (q,s) = 1,\ 2 < q < s \le n. 
\end{equation}
If $\alpha = 0$, then \( \mathsf{P}_0^\mathsf{I}(t) = t^{s-q} (t^q - 1)\) and $\mathsf{P}_0^\mathsf{I}$ has zeros 
\[ \overbrace{0,\ldots,0}^{s-q}, 1, \omega_q,\ldots,\omega_q^{q-1}. \]
If $\alpha = 1$, then $\mathsf{P}_1^\mathsf{I}(t) = t^s - 1$ and $\mathsf{P}_1^\mathsf{I}$ has zeros \( 1,\omega_s,\ldots,\omega_s^{s-1} \).

\begin{lemma}[Euclid's Lemma]
    \thlabel{littlelemma}
        If $\gcd(a,b) = 1$ and $a \mid b c$, then $a \mid c$.
\end{lemma}


\begin{theorem}[Forbidden rays]
    \thlabel{sqangle}
        Let $k \in \mathbb{Z}$, $\rho \in [0,\infty)$, and $\alpha \in (0,1)$. If $\Im \left(\omega_{2s}^k\right) \ne 0$, then $\mathsf{P}_\alpha^\mathsf{I} (\rho \omega_{2s}^k)  \neq 0$. Similarly, if $\Im\left(\omega_{2q}^k\right) \ne 0$, then $\mathsf{P}_\alpha^\mathsf{I} (\rho \omega_{2q}^k)  \neq 0$. 
\end{theorem}

\begin{proof}
Since $\omega_{2s}^k$ is a zero of $t^{2s} - 1 = (t^s + 1)(t^s - 1)$, it follows that $(\omega_{2s}^k)^s = \omega_{2s}^{ks} =\pm 1$. For contradiction, if 
\begin{align*}
    0 = \mathsf{P}_\alpha^\mathsf{I} (\rho \omega_{2s}^k)  
    &= \rho^s \omega_{2s}^{ks} - \beta \left(\rho^{s-q} \omega_{2s}^{ks-kq} \right) - \alpha                                                            \\
    &= \pm \rho^s \mp \beta \rho^{s-q} \omega_{2s}^{-kq} - \alpha                                                                                       \\
    &= \left[ \pm \rho^s \mp \beta \rho^{s-q}\cos (\frac{-\pi kq}{s}) - \alpha \right] + \ii\left[\mp \beta \rho^{s-q}\sin(\frac{-\pi k q}{s}) \right]  \\
    &= \left[ \pm \rho^s \mp \beta \rho^{s-q}\cos (\frac{\pi kq}{s}) - \alpha \right] + \ii\left[\pm\beta \rho^{s-q}\sin(\frac{\pi k q}{s}) \right],
\end{align*}
then
\begin{equation}
    \pm \rho^s \mp \beta \rho^{s-q}\cos (\frac{\pi kq}{s}) - \alpha = 0 \label{realpart}    
\end{equation}
and 
\begin{equation}
    \beta \rho^{s-q}\sin(\frac{\pi k q}{s}) = 0. \label{imagpart}
\end{equation}
If $\rho = 0$, then, by \eqref{realpart}, $\alpha = 0$, a contradiction. If $\rho \ne 0$, then, by \eqref{imagpart}  
    \[ \sin(\frac{\pi k q}{s}) = 0. \]
Thus, $s \mid k q$ and since $\gcd(q,s) = 1$, by \thref{littlelemma}, it follows that $s \mid k$. Consequently,  
\begin{align*}
    \Im \left( \omega_{2s}^k \right) = \sin(\frac{\pi k}{s}) = 0,
\end{align*}
a contradiction.
 
Since $\omega_{2q}^k$ is a zero of $t^{2q} - 1 = (t^q + 1)(t^q - 1)$, it follows that $(\omega_{2q}^k)^q = \pm 1$ and $(\omega_{2q}^k)^{-q} = \pm 1$. For contradiction, if  
\begin{align*}
0 = \mathsf{P}_\alpha^\mathsf{I} (\rho \omega_{2s}^k) 
&= (\rho \omega_{2q}^k)^s - \beta (\rho \omega_{2q}^k)^{s-q} - \alpha                                                                                       \\
&= \rho^s \omega_{2q}^{ks} \pm \beta \rho^{s-q} \omega_{2q}^{ks} - \alpha                                                                                   \\
&= \left[\rho^{s-q} (\rho^q \pm \beta)\cos (\frac{\pi k s}{q}) - \alpha \right] + \ii\left[ \rho^{s-q}(\rho^q \pm \beta) \sin( \frac{\pi k s}{q}) \right],
\end{align*}
then,  
\begin{equation}
    \label{realpart2}
        \rho^{s-q} (\rho^q \pm \beta)\cos (\frac{\pi k s}{q}) - \alpha = 0
\end{equation}
and
\begin{equation}
    \label{imagpart2}
        \rho^{s-q}(\rho^q \pm \beta) \sin( \frac{\pi k s}{q}) = 0
\end{equation}
If $\rho = 0$ or $\rho^q \pm \beta = 0$, then, by \eqref{realpart2}, $\alpha = 0$ a contradiction. If $\rho\ne 0$ and $\rho^q \pm \beta \ne 0$, then 
\[ \sin (\frac{\pi k s}{q}) = 0. \]
by \eqref{imagpart2}. Thus, $q \mid ks$ and since $\gcd(q,s) = 1$, by \thref{littlelemma}, it follows that $q \mid k$. Consequently,  
\begin{align*}
	\Im \left( \omega_{2q}^k \right) = \sin(\frac{\pi k}{q}) = 0,
\end{align*}
a contradiction.
\end{proof}

\begin{observation}
    \thlabel{obs:farey}
        \begin{enumerate}
            [label=\roman*)]
            \item If \(\frac{p}{q} < \frac{r}{s}\), then 
            \[ \frac{r-1}{s} < \frac{2r-1}{2s} <  \frac{p}{q} < \frac{r}{s} < \frac{2r+1}{2s} < \frac{r+1}{s} \Longleftrightarrow q > 2. \]
        
            \item If \(\frac{r}{s} < \frac{p}{q}\), then 
            \[ \frac{r-1}{s} < \frac{2r-1}{2s} <  \frac{r}{s} < \frac{p}{q} < \frac{2r+1}{2s} < \frac{r+1}{s} \Longleftrightarrow q > 2. \]
        \end{enumerate}
\end{observation}

\begin{proof} 
    \begin{enumerate}
        [label=\roman*)]
        \item In this case, $qr-ps=1$ and it suffices to show that 
            \[\frac{2r-1}{2s} < \frac{p}{q} \]
            given that the inequalities 
            \[ \frac{r-1}{2s} < \frac{2r-1}{2s} \]
            and 
            \[ \frac{r}{s} < \frac{2r+1}{2s} < \frac{r+1}{s} \]
            are obvious; to this end, notice that
            \[\frac{2r-1}{2s} < \frac{p}{q} \Longleftrightarrow 2(qr-ps) < q \Longleftrightarrow 2 < q. \]
    
        \item \( \frac{r}{s} < \frac{p}{q} \). In this case, $ps - qr = 1$ and it suffices to show that 
            \[ \frac{p}{q} < \frac{2r+1}{2s} \]
            given that the inequalities 
            \[ \frac{r-1}{2s} < \frac{2r-1}{2s} < \frac{r}{s} \]
            and 
            \[ \frac{2r+1}{2s} < \frac{r+1}{s} \]
            are obvious; to this end, notice that
            \[ \frac{p}{q} < \frac{2r+1}{2s}  \Longleftrightarrow 2(ps - qr) < q \Longleftrightarrow 2 < q. \qedhere \]
    \end{enumerate}
\end{proof}

\begin{observation}
    \thlabel{obs:farey2}
        \begin{enumerate}[label=\roman*)]
            \item If $p/q < r/s$, then 
                \[ \frac{r}{s} < \frac{2p+1}{2q} \Longleftrightarrow 2 < s. \]
    
            \item If $r/s < p/q$, then
                \[ \frac{2p-1}{2q} < \frac{r}{s} \Longleftrightarrow 2 < s. \]
        \end{enumerate}
\end{observation}

\begin{proof} 
    \begin{enumerate}
        [label=\roman*)]
        \item In this case, $qr - ps = 1$ and 
            \[ \frac{r}{s} < \frac{2p+1}{2q} \Longleftrightarrow 2(qr - ps) < s \Longleftrightarrow 2 < s. \]
    
        \item In this case, $ps - qr = 1$ and 
            \[ \frac{2p-1}{2q} < \frac{r}{s} \Longleftrightarrow 2(ps - qr) < s\Longleftrightarrow 2 < s. \qedhere \]
    \end{enumerate}
\end{proof}

\begin{theorem}
    \thlabel{mainresult}
        If $\mathsf{P}_\alpha^\mathsf{I}$ is the polynomial defined in \eqref{type_one_ito}, then there is a continuous function $\lambda: [0,1] \longrightarrow \mathbb{C}$ such that $\lambda(0) = \omega_q^p$, $\lambda(1) = \omega_{s}^{r}$, and $\mathsf{P}_\alpha^\mathsf{I} (\lambda(\alpha)) = 0, \forall \alpha \in [0,1]$. Furthermore, 
        \[ \min \left\{ \frac{p}{q}, \frac{r}{s} \right\} \le \frac{\Arg \lambda(\alpha)}{2\pi} \le \max \left\{ \frac{p}{q}, \frac{r}{s} \right\}, \forall \alpha \in (0,1). \]   
\end{theorem}  

\begin{proof}
    By \thref{kato}, there are $s$ continuous functions $\lambda_j: [0,1] \longrightarrow \mathbb{C}$, $j=1,\ldots,s$, such that $\mathsf{P}_\alpha^\mathsf{I} (\lambda_j(\alpha)) = 0$, $\forall \alpha \in [0,1]$. Notice that $\exists k \in \{1,\ldots,s\}$ such that $\lambda_k (0) = \omega_q^p$. By \thref{sqangle}, \thref{obs:farey} and \thref{obs:farey2}, it must be the case that $\lambda_k (1) = \omega_s^r$ (otherwise, $\lambda_k$ would cross a ``forbidden'' ray). Furthermore, the arc must be entirely contained in the stated sector given that, again, it would cross either of the forbidden rays $\{ \rho \omega_q^p \mid \rho \ge 0 \}$ and $\{ \rho \omega_s^r \mid \rho \ge 0 \}$ (see Figure \ref{fig:schematics}).  
\end{proof}

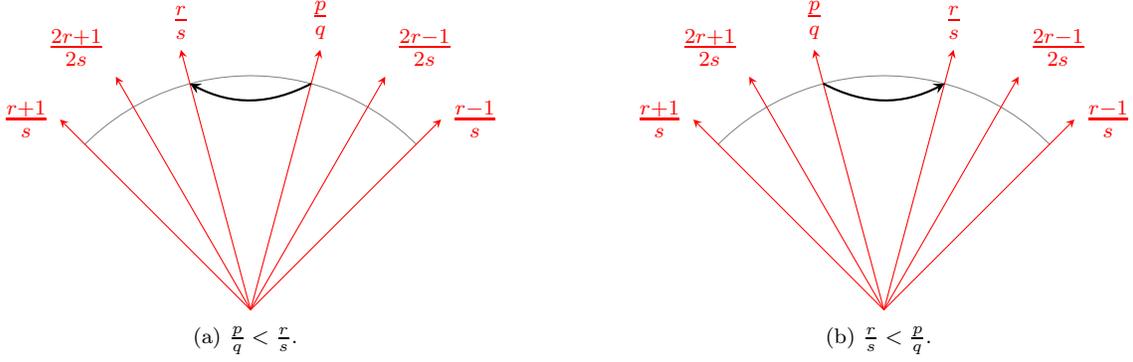
\begin{figure}[H]
    \centering
        \begin{subfigure}{.49\textwidth}
            \centering
            \scalebox{1}{
            \begin{tikzpicture}
                \begin{axis}[
                    axis lines=none,
                    axis equal image,
                    xmin=-1.1,
                    xmax=1.1,
                    ymin=0.0,
                    ymax=1.4,]
                    \draw[gray,very thin,domain=45:135] plot ({cos(\x)}, {sin(\x)});
                    
                    \draw[-stealth,red] (0,0) -- ({1.15*cos(45)}, {1.15*sin(45)}) node[right]                 {$\frac{r-1}{s}$};
                    \draw[-stealth,red] (0,0) -- ({1.15*cos(60)}, {1.15*sin(60)}) node[above right]           {$\frac{2r-1}{2s}$};
                    \draw[-stealth,red] (0,0) -- ({1.15*cos(75)}, {1.15*sin(75)}) node[above]                 {$\frac{p}{q}$};
                    \draw[-stealth,red] (0,0) -- ({1.15*cos(105)}, {1.15*sin(105)}) node[above]               {$\frac{r}{s}$};
                    \draw[-stealth,red] (0,0) -- ({1.15*cos(120)}, {1.15*sin(120)}) node[above left]          {$\frac{2r+1}{2s}$};
                    \draw[-stealth,red] (0,0) -- ({1.15*cos(135)},{1.15*sin(135)}) node[left]                 {$\frac{r+1}{s}$};
            
                    \draw[-stealth, thick] ({cos(75)},{sin(75)}) to [bend left=27.5] ({cos(105)},{sin(105)});
                \end{axis}
            \end{tikzpicture}}
            \caption{$\frac{p}{q} < \frac{r}{s}$.}
            \label{fig:schematic_one}
        \end{subfigure}
    \hfill    
        \begin{subfigure}{.49\textwidth}
            \centering
                \scalebox{1.00}{
                \begin{tikzpicture}
                    \begin{axis}[
                    axis lines=none,
                    axis equal image,
                    xmin=-1.1,
                    xmax=1.1,
                    ymin=0.0,
                    ymax=1.4,]
                    \draw[gray,very thin,domain=45:135] plot ({cos(\x)}, {sin(\x)});
                    
                    \draw[-stealth,red] (0,0) -- ({1.15*cos(45)}, {1.15*sin(45)}) node[right]                 {$\frac{r-1}{s}$};
                    \draw[-stealth,red] (0,0) -- ({1.15*cos(60)}, {1.15*sin(60)}) node[above right]           {$\frac{2r-1}{2s}$};
                    \draw[-stealth,red] (0,0) -- ({1.15*cos(75)}, {1.15*sin(75)}) node[above]                 {$\frac{r}{s}$};
                    \draw[-stealth,red] (0,0) -- ({1.15*cos(105)}, {1.15*sin(105)}) node[above]               {$\frac{p}{q}$};
                    \draw[-stealth,red] (0,0) -- ({1.15*cos(120)}, {1.15*sin(120)}) node[above left]          {$\frac{2r+1}{2s}$};
                    \draw[-stealth,red] (0,0) -- ({1.15*cos(135)},{1.15*sin(135)}) node[left]                 {$\frac{r+1}{s}$};
            
                    \draw[-stealth, thick] ({cos(105)},{sin(105)}) to [bend right=27.5] ({cos(75)},{sin(75)});
                    \end{axis}
                \end{tikzpicture}}
                \caption{$\frac{r}{s} < \frac{p}{q}$.}
            \label{fig:schematic_two}
        \end{subfigure}
    \caption{Schemata for Type I K-arcs.}
    \label{fig:schematics}
\end{figure}

\begin{lemma}
    \thlabel{type_one_mod}
        If $\mathsf{P}_\alpha^\mathsf{I}(\lambda) = 0$, with $\alpha \in (0,1)$, then $\vert \lambda \vert < 1$.    
\end{lemma}

\begin{proof}
    By Proposition \ref{modlessthanone}, $\vert \lambda \vert \le 1$ and since the companion matrix of $\mathsf{P}_\alpha^\mathsf{I}$ is nonnegative and primitive \cite[Theorem 3.2]{jp2017}, the result follows. 
\end{proof}

\begin{theorem}[c.f.~Kirkland et al.~{\cite[Corollary 4.2]{kls2020}}]
    If $\floor{n/q} = 1$, then the K-arc 
    \[ K_n\left( \min \left\{ \frac{p}{q}, \frac{r}{s} \right\}, \max \left\{ \frac{p}{q}, \frac{r}{s} \right\} \right) \]
    is simple.
\end{theorem}

\begin{proof}
    For contradiction, if $\lambda(\alpha) = \lambda = \lambda(\hat{\alpha})$ and $\mathsf{P}_\alpha(\lambda) = 0 = \mathsf{P}_{\hat{\alpha}}(\lambda)$, with $\alpha \ne \hat{\alpha}$, then $\beta \ne \hat{\beta}$ and
    \[ 
    \lambda^s - \beta \lambda^{s-q} - \alpha =  
    \lambda^s - \hat{\beta} \lambda^{s-q} - \hat{\alpha} \implies \lambda^{s-q} = 
    \frac{\hat{\alpha} - \alpha}{\beta - \hat{\beta}}. \]
    Since $\beta - \hat{\beta} = (1-\alpha)-(1-\hat{\alpha}) = \hat{\alpha} - \alpha$, it follows that $\lambda^{s-q} = 1$. Consequently, $\vert \lambda \vert = 1$. However, if $\alpha \in (0,1)$ or $\hat{\alpha} \in (0,1)$, then, by \thref{type_one_mod}, $\vert \lambda \vert < 1$, a contradiction. Othwerise, if $\alpha = 0$ and $\hat{\alpha} = 1$, then $\omega_q^p = \lambda(0) = \lambda(1) = \omega_s^r$, a contradiction since $\omega_q^p \ne \omega_s^r$. The case when $\alpha = 1$ and $\hat{\alpha} = 0$ yields a similar contradiction.   
\end{proof}

\section{Powers of K-Arcs}\label{sect:karcpowers}

The purpose of this section is to simplify and expand results given by Johnson and Paparella \cite[Theorem 5.4]{jp2017} and by Kim and Kim \cite[Theorem 4]{kk2020}) concerning point-wise powers of special Type I K-arcs. For similar results, see Joshi et al.~\cite{jks2023}.

Given $\lambda:[0,1] \longrightarrow \mathbb{C}$ and $d \in \mathbb{N}$, denote by $\lambda^d$ the function $\alpha \in [0,1] \longmapsto \lambda(\alpha)^d \in \mathbb{C}$. 

In this section, $d$, $m$, and $n$ are integers such that $1 < d < m \le n$ and $\left( \frac{1}{m}, \frac{1}{m-1} \right)$ is a Farey pair of order $n$ (as a consequence of \thref{leveque}, the latter holds if and only if $n < 2m -1$). Notice that the reduced Ito polynomial is given by 
\[ \mathsf{P}_\alpha^\mathsf{I} (t) = t^m - \beta t - \alpha,\ \alpha \in [0,1],\ \beta \coloneqq 1 - \alpha. \]
Let $\lambda: [0,1] \longrightarrow \mathbb{C}$ be the continuous function guaranteed by \thref{mainresult} such that \( \mathsf{P}_\alpha^\mathsf{I} (\lambda(\alpha)) = 0 \) ($\forall \alpha \in [0,1]$), $\lambda(0) = \omega_{m-1}$, $\lambda(1) = \omega_{m}$, and $\arg \lambda (\alpha) \in \left[ \frac{2\pi}{m}, \frac{2\pi}{m-1} \right]$ ($\forall \alpha \in [0,1]$).

The following result is also a facile consequence of \thref{leveque}. 

\begin{lemma}
    \thlabel{lem:poweredfareypairs}
        Suppose that $d$, $m$, and $n$ are integers such that $1 < d < m \le n$.
        \begin{enumerate}
            [label=\roman*)]
                \item If $m = dk$, then $\left( \frac{1}{m}, \frac{1}{m-1} \right)$ and $\left( \frac{1}{k}, \frac{d}{m-1} \right)$ are Farey pairs of order $n$ if and only if $n < m + k - 1$.
            
                \item If $m-1 = dk$, then $\left( \frac{1}{m}, \frac{1}{m-1} \right)$ and $\left( \frac{d}{m}, \frac{1}{k} \right)$ are Farey pairs of order $n$ if and only if $n < m + k$.
        \end{enumerate}
\end{lemma}

\begin{theorem}
    \thlabel{thm:mink_power}
        Suppose that $d$, $m$, and $n$ are integers such that $1 < d < m \le n$.
        \begin{enumerate}
            [label=\roman*)]
                \item \label{partone} If $m = dk$ and $n < m + k - 1$, then part \ref{secondstep} of \thref{bigprop} holds for $\left( \frac{1}{k}, \frac{d}{m-1} \right)$.
            
                \item If $m - 1 = dk$ and $n < m + k - 1$, then part \ref{secondstep} of \thref{bigprop} holds for $\left( \frac{d}{m}, \frac{1}{k} \right)$.
        \end{enumerate}    
\end{theorem}

\begin{proof} 
(i) By \thref{lem:poweredfareypairs}, $\left( \frac{1}{m}, \frac{1}{m-1} \right)$ and $\left( \frac{1}{k}, \frac{d}{m-1} \right)$ are Farey pairs of order $n$. Since $k > 1$ and  
\begin{equation*}
    d = \frac{m}{k} \le \frac{n}{k} < \frac{m + k - 1}{k} = d + 1 - \frac{1}{k} < d + 1,
\end{equation*}
it follows that $\floor{n/k} = d$. Thus, the Ito equation for $\left( \frac{1}{k}, \frac{d}{m-1} \right)$ is given by
\begin{align*}
    t^{m-1}(t^k - \delta)^d = \gamma^d t^m,\ \gamma \in [0,1],\ \delta \coloneqq 1 -\gamma,
\end{align*}
and any nonzero root satisfies the reduced Ito equation
\begin{equation}
    (t^k - \delta)^d = \gamma^d t,\ \gamma \in [0,1],\ \delta \coloneqq 1 -\gamma. \label{reditopower}
\end{equation}    
If $\lambda_\delta \coloneqq \lambda(\delta)$, then $\lambda_\delta^m - \delta = \gamma \lambda_\delta$ and 
\[ ((\lambda_\delta^d)^k - \delta)^d = (\lambda_\delta^m - \delta)^d = (\gamma \lambda_\delta)^d = \gamma^d \lambda_\delta^d, \] 
i.e., $\lambda_\delta^d$ satisfies \eqref{reditopower}. Thus, the function $\lambda^d$ is the required function.

(ii) Suppose that $m-1 = dk$, $\left( \frac{1}{m}, \frac{1}{m-1} \right)$ and $\left( \frac{d}{m}, \frac{1}{k} \right)$ are Farey pairs of order $n$, and $n < m + k - 1$. Since $k > 1$ and  
\begin{equation*}
    d = \frac{m-1}{k} < \frac{n}{k} < \frac{m + k - 1}{k} = d + \frac{1}{k} < d + 1,
\end{equation*}
it follows that $\floor{n/k} = d$. Thus, the Ito equation for $\left( \frac{d}{m}, \frac{1}{k} \right)$ is given by
\begin{align*}
    t^m(t^k - \delta)^d = \gamma^d t^{m-1},\ \gamma \in [0,1],\ \delta \coloneqq 1 -\gamma,
\end{align*}
and any nonzero root satisfies the reduced Ito equation
\begin{equation}
    t(t^k - \delta)^d = \gamma^d,\ \gamma \in [0,1],\ \delta \coloneqq 1 -\gamma. \label{reditopower2}
\end{equation}    
If $\lambda_\gamma \coloneqq \lambda(\gamma)$, then $\gamma = \lambda_\gamma(\lambda_\gamma^{m-1} - \delta)$ and 
\[ \gamma^d = (\lambda_\gamma(\lambda_\gamma^{m-1} - \delta))^d = \lambda_\gamma^d ((\lambda_\gamma^d)^k - \delta)^d, \] 
i.e., $\lambda_\gamma^d$ satisfies \eqref{reditopower2}. Thus, the function $\lambda^d$ is the required function.
\end{proof}

\begin{remark}
    \thlabel{rem:special_arc}
        If $n = 2\ell$, with $\ell \ge 2$, then the reduced Ito polynomial for $\left( \frac{1}{2\ell}, \frac{1}{2\ell - 1} \right)$ is given by 
            \[ \mathsf{P}_\alpha^\mathsf{I} (t) = t^{2\ell} - \beta t - \alpha,\ \alpha \in [0,1],\ \beta \coloneqq 1 - \alpha. \]
        Let $\lambda: [0,1] \longrightarrow \mathbb{C}$ be the continuous function guaranteed by \thref{mainresult} such that \( \mathsf{P}_\alpha^\mathsf{I} (\lambda(\alpha)) = 0 \) ($\forall \alpha \in [0,1]$), $\lambda(0) = \omega_{2\ell-1}$, $\lambda(1) = \omega_{2\ell}$, and 
        \[ \arg \lambda(\alpha) \in 2\pi\left[ \frac{1}{2\ell}, \frac{1}{2\ell-1} \right] (\forall \alpha \in [0,1]). \] 
        By part \ref{partone} of \thref{thm:mink_power}, $\lambda^\ell$ satisfies part \ref{secondstep} of \thref{bigprop} for $\left( \frac{1}{2}, \frac{\ell}{2\ell - 1} \right)$. In view of \thref{realsymm}, $\overline{\lambda^\ell}$ satisfies part \ref{secondstep} of \thref{bigprop} for $\left( \frac{\ell - 1}{2\ell - 1}, \frac{1}{2} \right)$.
\end{remark}

\begin{remark}
    \thlabel{rem:special_arc2}
        If $n = 4\ell$, with $\ell \ge 1$, then the reduced Ito polynomial for $\left( \frac{2\ell - 1}{n}, \frac{\ell}{2\ell + 1} \right)$ is given by 
            \[ \mathsf{P}_\alpha^\mathsf{I} (t) = t^{4\ell} - \beta t^{2\ell - 1} - \alpha,\ \alpha \in [0,1],\ \beta \coloneqq 1 - \alpha. \]
        Let $\lambda: [0,1] \longrightarrow \mathbb{C}$ be the continuous function guaranteed by \thref{mainresult} such that \( \mathsf{P}_\alpha^\mathsf{I} (\lambda(\alpha)) = 0 \) ($\forall \alpha \in [0,1]$), $\lambda(0) = \omega_{2\ell + 1}^\ell$, $\lambda(1) = \omega_{4\ell}^{2\ell - 1}$, and 
        \[ \arg \lambda(\alpha) \in 2\pi\left[ \frac{2\ell - 1}{n}, \frac{\ell}{2\ell + 1} \right] (\forall \alpha \in [0,1]).\] 
        By part \ref{partone} of \thref{thm:mink_power}, $\lambda^2$ satisfies part \ref{secondstep} of \thref{bigprop} for $\left( \frac{2\ell - 1}{2\ell}, \frac{2\ell}{2\ell + 1} \right)$. In view of \thref{realsymm}, $\overline{\lambda^\ell}$ satisfies part \ref{secondstep} of \thref{bigprop} for $\left( \frac{1}{2\ell}, \frac{1}{2\ell + 1} \right)$.
\end{remark}

\begin{example}
If $n=8$, then there are eleven K-arcs in the upper-half plane. Figure \ref{fig:enter-label} displays the K-arcs (solid) of $\Theta_8$ in the upper-half plane guaranteed by \thref{mainresult}, \thref{thm:mink_power}, \thref{rem:special_arc}, and \thref{rem:special_arc2}.  

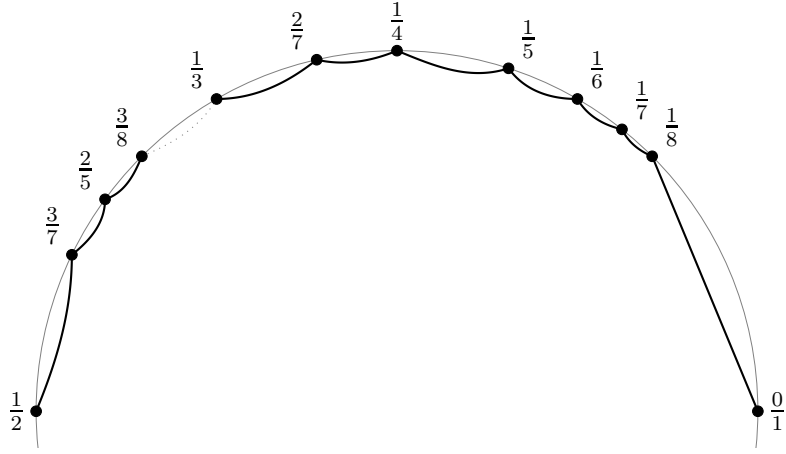
\begin{figure}[H]
    \centering
        \begin{tikzpicture}
        \begin{axis}[
            axis lines = none,
            axis equal image,
            scale=1.75,
            xmin=-1.25,
            xmax=1.25,
            ymin=-.1,
            ymax=1.25,
        ]
        \draw[very thin, gray] (0,0) circle(1.0);
        \addplot[thick] coordinates{(1,0) (0.7071067811865476, 0.7071067811865475)};    
        \addplot[thick] table {plots/arc1_theta8.dat};                                  
        \addplot[thick] table {plots/arc2_theta8.dat};                                  
        \addplot[thick] table {plots/arc3_theta8.dat};                                  
        \addplot[thick] table {plots/arc4_theta8.dat};                           
        \addplot[thick] table {plots/arc5_theta8.dat};                                  
        \addplot[thick] table {plots/arc6_theta8.dat};                                  
        \addplot[dotted, gray] table {plots/arc7_theta81.dat};  
        \addplot[dotted, gray] table {plots/arc7_theta82.dat};                        
        \addplot[thick] table {plots/Arc8_theta8.dat};                                  
        \addplot[thick] table {plots/arc9_theta8.dat};                                  
        \addplot[thick] table {plots/arc10_theta8.dat};                                 

        \filldraw[black] (1,0)                          circle (2pt) node[right]        {$\frac{0}{1}$};
        \filldraw[black] ({cos(45)}, {sin(45)})         circle (2pt) node[above right]  {$\frac{1}{8}$};
        \filldraw[black] ({cos(360/7)}, {sin(360/7)})   circle (2pt) node[above right]  {$\frac{1}{7}$};
        \filldraw[black] ({cos(60)}, {sin(60)})         circle (2pt) node[above right]  {$\frac{1}{6}$};
        \filldraw[black] ({cos(72)}, {sin(72)})         circle (2pt) node[above right]  {$\frac{1}{5}$};
        \filldraw[black] ({cos(90)}, {sin(90)})         circle (2pt) node[above]        {$\frac{1}{4}$};
        \filldraw[black] ({cos(720/7)}, {sin(720/7)})   circle (2pt) node[above left]   {$\frac{2}{7}$};
        \filldraw[black] ({cos(120)}, {sin(120)})       circle (2pt) node[above left]   {$\frac{1}{3}$};
        \filldraw[black] ({cos(135)}, {sin(135)})       circle (2pt) node[above left]   {$\frac{3}{8}$};
        \filldraw[black] ({cos(144)}, {sin(144)})       circle (2pt) node[above left]   {$\frac{2}{5}$};
        \filldraw[black] ({cos(1080/7)}, {sin(1080/7)}) circle (2pt) node[above left]   {$\frac{3}{7}$};
        \filldraw[black] ({cos(180)}, {sin(180)})       circle (2pt) node[left]         {$\frac{1}{2}$};
        
        \end{axis}
    \end{tikzpicture}
    \caption{The region $\partial\Theta_8$ in the upper-half plane.}
    \label{fig:enter-label}
\end{figure}
\end{example}

\section{Boundary \& Extremal Points}

In this section, we provide an elementary proof that $\partial\Theta_n \subseteq E_n$ whenever $n > 3$.

The following result is well-known or otherwise easy to establish. 

\begin{lemma}
    \thlabel{lambdaconv}
        If $\lambda \in \Theta_n$, then $C_p(\lambda) \coloneqq \conv(1,\lambda,\ldots,\lambda^p) \subseteq \Theta_n,\forall p \in \mathbb{N}$.     
\end{lemma}

\begin{remark}
    \thlabel{dubucmalik}
    More can be asserted: if 
    $$\lambda \in \{ z \in \mathbb{C} \mid \vert{z}\vert < 1 \}\backslash \{ z \in \mathbb{C} \mid (\Re z > 0)\wedge (\Im z = 0)\},$$ 
    then there is a positive integer $q$ such that $\lambda^q \subseteq C_{q-1}(\lambda)$ (see Dubuc and Malik \cite[Theorem 2.2]{dm1992}). Furthermore, $q > 2$ whenever $\Im z \ne 0$. 
\end{remark}

\begin{theorem}
    If $n > 3$, then $\partial\Theta_n \subseteq E_n$.
\end{theorem}     

\begin{proof}
    Since $\partial\Theta_n \cap \{ z \in \mathbb{C} \mid \Im z = 0 \} = \{-1,1\}$, it suffices to show that $\lambda \in E_n$ whenever $\lambda \in \partial\Theta_n$ and $\Im \lambda \ne 0$. 
    
    To this end, let $\lambda = a + \mathsf{i} b \in \partial\Theta_n$ and suppose, for contradiction, that $\exists \gamma \in (1,\infty)$ such that $\gamma\lambda\in\Theta_n$.  By \thref{realsymm}, it can be assumed, without loss of generality, that $b = \Im \lambda > 0$. Following \thref{lambdaconv}, $C_p(\gamma \lambda) \subseteq \Theta_n,\forall p \in \mathbb{N}$ and following \thref{dubucmalik}, there is a positive integer $q > 2$ such that $\gamma^q\lambda^q \in C_{q-1}(\gamma \lambda)$. We now demonstrate that that $\lambda$ lies in the interior of $C_{q-1}(\gamma \lambda)$. 
 
    Recall that the distance from a point $(x_0,y_0) \in \mathbb{R}^2$ to the line passing through the points $P_1 \coloneqq (x_1,y_1)$ and $P_2 \coloneqq (x_2,y_2)$ is given by
    \[ d \left( \overleftrightarrow{P_1 P_2},(x_0,y_0) \right) = \frac{\vert (x_2 - x_1)(y_1 - y_0) - (x_1 - x_0)(y_2 - y_1) \vert}{\vert\vert P_1 - P_2 \vert\vert_2}. \]

    Several tedious, but straightforward calculations reveals that: 
    \begin{itemize}
        \item if $(x_0,y_0) = (a,b)$, $P_1 = (1,0)$, and $P_2 = (\gamma a,\gamma b)$, then
    \[ d_1 \coloneqq d \left( \overleftrightarrow{P_1 P_2},(x_0,y_0) \right) = \frac{b(\gamma - 1)}{\vert\vert P_1 - P_2 \vert\vert_2} > 0; \]

        \item if $(x_0,y_0) = (a,b)$, $P_1 = (\gamma a,\gamma b)$, and $P_2 = (\gamma^2 (a^2 - b^2), 2\gamma^2 ab)$, then
    \[ d_2 \coloneqq d \left(\overleftrightarrow{P_1 P_2},(x_0,y_0)\right) = \frac{\gamma^2(\gamma - 1)b(a^2 + b^2)}{\vert\vert P_1 - P_2 \vert\vert_2} > 0. \]
    \end{itemize}
    
    Thus, the distance from  $\lambda$ to each of the line segments $\conv(1, \gamma \lambda)$ and $\conv(\gamma \lambda, \gamma^2 \lambda^2)$ of the polygon $C_{q-1}(\gamma \lambda)$ is positive. Thus, we may select $\varepsilon > 0$ small enough such that $N_\varepsilon (\lambda) \coloneqq \{ z \in \mathbb{C} \mid \vert z - \lambda \vert < \varepsilon \} \subset C_{q-1}(\gamma \lambda) \subseteq \Theta_n$, a contradiction.
\end{proof}

\begin{remark}
    The argument above fails for $n=2$ and $n=3$ because, if $b = 0$, then $d_1 = d_2 = d_3 = 0$. 
\end{remark}

\section{Implications for Further Inquiry}

In addition to establishing part \ref{secondstep} of \thref{bigprop} for Type II and Type III K-arcs which are not point-wise powers of special Type I arcs, proving the following conjecture, which is part \ref{thirdstep} of \thref{bigprop}, would complete the proof of \thref{karpito}. 

\begin{conjecture}
    Suppose that $n > 3$ and $p/q$ and $r/s$ are Farey neighbors. If $\lambda$ is a function satisfying the properties listed in part \ref{secondstep} of \thref{bigprop}, then $\lambda(\alpha) \in E_n$, $\forall \alpha \in [0,1]$.
\end{conjecture}

\section*{Acknowledgement}

We thank the anonymous referee for their careful review and suggestions that improved this work.  
     
\bibliographystyle{abbrv}
\bibliography{karparcs}

\end{document}